\documentclass[11pt,a4paper]{amsart}
\usepackage[margin=25mm]{geometry}
\usepackage[numbers]{natbib}
\usepackage{hyperref}

\usepackage{amssymb,amsmath}
\usepackage{amsthm}
\usepackage{bbm, bm}
\usepackage{stmaryrd}
\usepackage[usenames,dvipsnames]{xcolor}
\usepackage{color}
\usepackage{graphicx}
\usepackage{caption}
\usepackage{float}
\usepackage{wrapfig}

\usepackage{lineno}

\usepackage{abstract}

\input xy
\xyoption{all} 

\numberwithin{equation}{section}

\newtheorem{theorem}{Theorem}[section]

\newtheorem{lemma}[theorem]{Lemma}
\newtheorem{proposition}[theorem]{Proposition}

\newtheorem{observation}[theorem]{Observation}
\theoremstyle{definition}
\newtheorem{definition}[theorem]{Definition}
\newtheorem{example}[theorem]{Example}
\newtheorem{coexample}[theorem]{Counter Example}

\newenvironment{remark}[1][Remark.]{\begin{trivlist}
\item[\hskip \labelsep {\bfseries #1}]  }{ \end{trivlist}}

\newcommand{\Id}{\mathbbmss{1}}

\DeclareMathOperator{\Vect}{Vect}

\font\black=cmbx10 \font\sblack=cmbx7 \font\ssblack=cmbx5 \font\blackital=cmmib10  \skewchar\blackital='177
\font\sblackital=cmmib7 \skewchar\sblackital='177 \font\ssblackital=cmmib5 \skewchar\ssblackital='177
\font\sanss=cmss10 \font\ssanss=cmss8 
\font\sssanss=cmss8 scaled 600 \font\blackboard=msbm10 \font\sblackboard=msbm7 \font\ssblackboard=msbm5
\font\caligr=eusm10 \font\scaligr=eusm7 \font\sscaligr=eusm5  \font\fraktur=eufm10
\font\sfraktur=eufm7 \font\ssfraktur=eufm5 
\font\bsymb=cmsy10 scaled\magstep2
\def\all#1{\setbox0=\hbox{\lower1.5pt\hbox{\bsymb
       \char"38}}\setbox1=\hbox{$_{#1}$} \box0\lower2pt\box1\;}
\def\exi#1{\setbox0=\hbox{\lower1.5pt\hbox{\bsymb \char"39}}
       \setbox1=\hbox{$_{#1}$} \box0\lower2pt\box1\;}

\def\tx#1{{\fam0\relax#1}}

\newfam\bifam
\textfont\bifam=\blackital \scriptfont\bifam=\sblackital \scriptscriptfont\bifam=\ssblackital

\newfam\blfam
\textfont\blfam=\black \scriptfont\blfam=\sblack \scriptscriptfont\blfam=\ssblack

\newfam\bbfam
\textfont\bbfam=\blackboard \scriptfont\bbfam=\sblackboard \scriptscriptfont\bbfam=\ssblackboard

\newfam\ssfam
\textfont\ssfam=\sanss \scriptfont\ssfam=\ssanss \scriptscriptfont\ssfam=\sssanss
\def\sss#1{{\fam\ssfam\relax#1}}

\newfam\clfam
\textfont\clfam=\caligr \scriptfont\clfam=\scaligr \scriptscriptfont\clfam=\sscaligr

\newfam\frfam
\textfont\frfam=\fraktur \scriptfont\frfam=\sfraktur \scriptscriptfont\frfam=\ssfraktur

\def\hpb#1{\setbox0=\hbox{${#1}$}
    \copy0 \kern-\wd0 \kern.2pt \box0}
\def\vpb#1{\setbox0=\hbox{${#1}$}
    \copy0 \kern-\wd0 \raise.08pt \box0}

\def\pmb#1{\setbox0\hbox{${#1}$} \copy0 \kern-\wd0 \kern.2pt \box0}
\def\pmbb#1{\setbox0\hbox{${#1}$} \copy0 \kern-\wd0
      \kern.2pt \copy0 \kern-\wd0 \kern.2pt \box0}
\def\pmbbb#1{\setbox0\hbox{${#1}$} \copy0 \kern-\wd0
      \kern.2pt \copy0 \kern-\wd0 \kern.2pt
    \copy0 \kern-\wd0 \kern.2pt \box0}
\def\pmxb#1{\setbox0\hbox{${#1}$} \copy0 \kern-\wd0
      \kern.2pt \copy0 \kern-\wd0 \kern.2pt
      \copy0 \kern-\wd0 \kern.2pt \copy0 \kern-\wd0 \kern.2pt \box0}
\def\pmxbb#1{\setbox0\hbox{${#1}$} \copy0 \kern-\wd0 \kern.2pt
      \copy0 \kern-\wd0 \kern.2pt
      \copy0 \kern-\wd0 \kern.2pt \copy0 \kern-\wd0 \kern.2pt
      \copy0 \kern-\wd0 \kern.2pt \box0}

\mathchardef\za="710B  
\mathchardef\zb="710C  
\mathchardef\zg="710D  
\mathchardef\zd="710E  
\mathchardef\zve="710F 
\mathchardef\zz="7110  
\mathchardef\zh="7111  
\mathchardef\zvy="7112 
\mathchardef\zi="7113  
\mathchardef\zk="7114  
\mathchardef\zl="7115  
\mathchardef\zm="7116  
\mathchardef\zn="7117  
\mathchardef\zx="7118  
\mathchardef\zp="7119  
\mathchardef\zr="711A  
\mathchardef\zs="711B  
\mathchardef\zt="711C  
\mathchardef\zu="711D  
\mathchardef\zvf="711E 
\mathchardef\zq="711F  
\mathchardef\zc="7120  
\mathchardef\zw="7121  
\mathchardef\ze="7122  
\mathchardef\zy="7123  
\mathchardef\zf="7124  
\mathchardef\zvr="7125 
\mathchardef\zvs="7126 
\mathchardef\zf="7127  
\mathchardef\zG="7000  
\mathchardef\zD="7001  
\mathchardef\zY="7002  
\mathchardef\zL="7003  
\mathchardef\zX="7004  
\mathchardef\zP="7005  
\mathchardef\zS="7006  
\mathchardef\zU="7007  
\mathchardef\zF="7008  
\mathchardef\zW="700A  
\mathchardef\zC="7009  

\newcommand{\be}{\begin{equation}}
\newcommand{\ee}{\end{equation}}

\newcommand{\bea}{\begin{eqnarray}}
\newcommand{\eea}{\end{eqnarray}}
\def\*{{\textstyle *}}
\newcommand{\R}{{\mathbb R}}

\newcommand{\Z}{{\mathbb Z}}

\newcommand{\s}{{\textstyle *}}

\def\Sec{\sss{Sec}}
\def\Vect{\sss{Vect}}

\def\sT{{\sss T}}

\def\xi{\tx{i}}

\def\s*{{\scriptstyle *}}

\newcommand{\beas}{\begin{eqnarray*}}
\newcommand{\eeas}{\end{eqnarray*}}

\title{Para-associative Algebroids} 
\author{Andrew James Bruce } 
   \email{andrewjamesbruce@googlemail.com}

   \date{\today}
 
\begin{document}
 \maketitle
\vspace{-20pt}
\begin{abstract}{\noindent We introduce para-associative algebroids as vector bundles whose sections form a ternary algebra with a generalised form of associativity.  We show that a necessary and sufficient condition for local triviality is the existence of a differential connection, i.e., a connection that satisfies the Leibniz rule over the ternary product. }\\

\noindent {\Small \textbf{Keywords:} Ternary Algebras;~ Vector Bundles;~Generalised Associativity}
\end{abstract}
\tableofcontents
\section{Introduction}
While it is true that physics and differential geometry are usually formulated in terms of binary operations, such as groups and (non-)associative algebras, mathematical structures with a ternary composition law have periodically entered the arena. Ternary algebras with a natural generalisation of associativity have appeared in mathematics and mathematical physics, a short list of works include \cite{Abramov:2009,Bazunova:2004,Bruce:2022,Carlsson:1976,Hestenes:1962,Kerner:2008,Kerner:2018,Zettl:1983}. The basic idea is to replace a binary product in the definition of an algebra with a ternary operation. There are two kinds of direct generalisation of the associative property, using the nomenclature of Kerner \cite{Kerner:2018}:
\begin{description}
 \setlength\itemsep{1em}
\item[A-associativity] $[[x_1, x_2, x_3], x_4, x_5] = [x_1,[x_2, x_3, x_4], x_5] = [x_1, x_2, [x_3, x_4, x_5] ]$,
\item[B-associativity] $[[x_1, x_2, x_3], x_4, x_5] = [x_1,[x_4, x_3, x_2], x_5] = [x_1, x_2, [x_3, x_4, x_5] ]$.
\end{description}
While A-associativity seems natural,  B-associativity is far more common when it comes to examples. The reason for this is that semiheap and heaps satisfy  B-associativity, which is more commonly known as para-associativity. We remark that  Pr\"{u}fer \cite{Prufer:1924} (1924) and independently  Baer \cite{Baer:1929} (1929) introduced the notion of a heap as a set with a ternary operation that satisfies the para-associativity property. For details of heaps, semiheaps and related structures the reader can consult Hollings \& Lawson \cite{Hollings:2017}. In this note we will restrict attention to ternary algebras that satisfy the para-associative property, we will refer to such algebras as para-associative algebras (see Definition \ref{def:ParaAlgbra}). Ternary groups, understood as a set with a ternary product satisfying type A-associativity were introduced by Dörnte \cite{Dörnte:1928} (1928). Quite general polyadic algebraic structures have been studied by Duplij \cite{Duplij:2022}. \par 
Horizontal categorification or oidification is the process of passing from a mathematical structure with a single object to a mathematical structure with multiple objects. A well-known example is that of a groupoid as a group with multiple objects, or multiple identities.  Loosely, passing from an algebra to an algebroid is replacing the underlying vector space with a module in the definition of an algebra. As we are interested in differential geometry, we restrict to $C^\infty(M)$-modules of sections of vector bundles over a smooth manifold $M$.  For example, a Lie algebroid can be seen (in part) as a Lie algebra in which the elements have been replaced by sections of a vector bundle (see \cite{Pradines:1967}). Similarly, an associative algebroid is, loosely, an associative algebra in which the elements are now sections of a vector bundle. As a multiple object generalisation of an algebra one need only observe that each fibre of an algebroid is an algebra, of the required type. For instance, the fibres of a Lie algebroid are all Lie algebras. Other examples in the literature include   Courant algebroids, Filippov algebroids, Jacobi algebroids and Leibniz algebroids (see \cite{Liu:1997,Grabowski:2000,Grabowski:2001,Ibáñez:1999}, respectively). In these examples, there is no local trivialisation condition, i.e., there is no reason why all the fibre algebras should be isomorphic. \par 
Associative algebra bundles, essentially associative algebroids with isomorphic fibre algebras, are common in physics and differential geometry. Tensor-algebra bundles, exterior bundles, symmetric bundles, Clifford bundles associated with a vector bundle, commonly the tangent bundle, are all well-known examples.  \par 
In this note, we consider para-associative algebroids as vector bundles whose space of section is a para-associative $C^\infty(M)$-algebra (see Definition \ref{def:ParaAlgBroid}). Such bundles are ternary objects in the category of vector bundles, but we will avoid extensive use of category theory and take a geometric approach. The reader may consult Mac Lane \cite{MacLane:1988} for an overview of category theory.  A canonical and guiding example of a para-associative algebroid, as we shall see, is the tangent bundle of a Riemannian manifold. \par 
We establish several results:
\begin{enumerate}
\setlength\itemsep{1em}
\item The para-associative algebras on the fibres are well-defined (Proposition \ref{prop:FibWellDef}).
\item By selecting a section of a para-associative algebroid, one can construct an associative algebroid (Proposition \ref{prop:ParaAssToAssAlg}).
\item If the para-associative algebroid can be equipped with a differential connection (see Definition \ref{def:DiffCon}), then we have local triviality, i.e., all the fibre para-associative algebras are isomorphic (see Theorem \ref{thm:DiffConBundle} and Proposition \ref{prop:ParaBunConAlways}). The existence of a differential connection is a necessary and sufficient condition for local triviality (see  \ref{obs:SuffNessCon}). These results generalise those of Canlubo \cite{Canlubo:2019}, who gave similar statements for the associative case.
\end{enumerate}
\textbf{Conventions:} By manifold we mean a smooth manifold that is real, finite-dimensional, Hausdorff, and second-countable. 
\section{Ternary Algebras on Vector Bundles}
\subsection{Para-associative Algebras}
For completeness, we will quickly recall the notion of a para-associative algebra. As we will be focusing on real geometry, we will restrict attention to real algebraic structures.
\begin{definition}\label{def:ParaAlgbra}
A \emph{(real) ternary para-associative algebra} consists of the following data:
\begin{enumerate}
\item A real vector space $\mathcal{A}$.
\item A ternary product 
     $$[-,-,-] : \mathcal{A} \times \mathcal{A} \times \mathcal{A} \longrightarrow \mathcal{A}\,,$$
     that 
     \begin{enumerate}
     \item is $\R$-linear in each entry, 
     \item is para-associative, i.e.,
     $$\big[ [x_1,x_2,x_3] , x_4,x_5 \big] = \big [ x_1,[x_4,x_3,x_2],x_5\big] =  \big[ x_1,x_2,[x_3 , x_4,x_5] \big]\, ,$$
     for all $x_i \in \mathcal{A}$.
     \end{enumerate}
\end{enumerate}
\end{definition}
\begin{remark}
In the complex setting, examples show us that we should take the ternary product as being conjugate-linear in the second argument.
\end{remark}
We will denote a (real) ternary para-associative algebra as a pair $(\mathcal{A}, [-,-,-])$ or just $\mathcal{A}$ if there is no need to explicitly present the ternary product. For brevity, we will generally refer to para-associative algebras.\par
\begin{definition}
Let $(\mathcal{A}, [-,-,-])$ and $(\mathcal{A}', [-,-,-]')$. A \emph{para-associative algebra homomorphism} is a linear map $\phi :\mathcal{A} \rightarrow \mathcal{A}'$ such that 
$$\phi [x_1,x_2,x_3] = [\phi(x_1),\phi(x_2),\phi(x_3)]\,,$$
for all $x_i \in \mathcal{A}$.
\end{definition}
In other words, a para-associative algebra is an `algebra' in which the binary product is replaced with a semiheap operation. We will say that a para-associative algebra is \emph{commutative} if $[x,y,z] = [z,y,x]$ for all elements.  An element $e$ of a para-associative algebra is said to be a \emph{biunit} if $[e,e,x]= [x,e,e] = x$ for all $x$.  If every element of a para-associative algebra is a biunit, then the ternary product is a heap operation. Due to $\R$-linearity, it is clear that $0$ is a (unique) absorbing element.  In general, there maybe (left/central/right) non-zero divisors, defined in the obvious way.  
\begin{example}\label{exa:ParaAssBiLin}
Consider a real vector space $\mathit{V}$ (not necessarily finite dimensional) equipped with a symmetric bilinear form $B :\mathit{V}  \times \mathit{V} \rightarrow \R $. We can construct a para-associative algebra on $\mathit{V}$ by defining 
$$[u,v,w]:= B(u,v)w\,,$$
for all $u,v$ and $w \in \mathit{V}$. Suppose we have another real vector space $\mathit{V}'$ equipped with a symmetric bilinear form $B'$. A linear map $\phi : \mathit{V}\rightarrow \mathit{V}'$, such that $B(u,v) = B'(\phi(u), \phi(v))$, is para-associative algebra homomorphism. The opposite para-associative algebra on $\mathit{V}$ is  defined as $[u,v,w]^{\mathrm{op}} := [w,v,u] = B(w,v)u = u B(v,w)$. If $B$ is degenerate, we have a non-trivial kernel 
$$\ker(B) := \left\{u \in \mathit{V}~|~ B(u,v)=0, ~\textnormal{for all}~v \in \mathit{V} \right\}\,.$$
Then any $u \in \ker(B)$ is a left and central divisor, i.e., $[u,v,w] =0$ and $[v,u,w]=0$ for all $v,w  \in \mathit{V}$.
\end{example}
\begin{coexample} 
The symmetry property of the bi-linear form in Example \ref{exa:ParaAssBiLin} is essential for the para-associative property to hold. Thus, if instead we consider an antisymmetric bi-linear form $B$ on a vector space $\mathit{V}$, and define $[u,v,w]:= B(u,v)w$, then the `middle' term of the para-associative law does not hold.
\end{coexample}
\begin{example}\label{exm:ParaAssR2}
Take $\R^2$ and its standard basis $\{e_1, e_2\}$. We can define a commutative heap on this basis by considering the heap associated with the cyclic group $C_2$ by setting $e_1 =  \Id$ and $e_2 e_2 = e_1$. The heap is then given by
\begin{align*}
& [e_1,e_1,e_1] = e_1\,, \qquad [e_1,e_1,e_2] = e_2\,, \qquad [e_2,e_1,e_1] = e_2\,,\qquad [e_2,e_1,e_2] = e_1\,,\\
& [e_1,e_2,e_1] = e_2\,, \qquad [e_1,e_2,e_2] = e_1\,, \qquad [e_2,e_2,e_1] = e_1\,,\qquad [e_2,e_2,e_2] = e_2\,.
\end{align*}
The para-associative algebra is then defined using this basis, i.e., 
$$[u,v,w] := u^iv^jw^k [e_k, e_j , e_i]\,,$$
where we have used the Einstein summation convention.  Note that choosing $e_2 = \Id$ and $e_1 e_1 = e_2$ gives the same heap. Thus, the linear map given by $\phi(e_1) = e_2$  and $\phi(e_2) = e_1$, is an isomorphism of para-associative algebras. As the para-associative algebra is commutative, $[u,v,w]^{\mathrm{op}} := [w,v,u] = [u,v,w]$.
\end{example}
\subsection{Para-associative Algebroids; Definition and First Examples}
Passing from para-associative algebra to an algebroid, we make the following natural definition.
\begin{definition}\label{def:ParaAlgBroid}
A \emph{para-associative algebroid} is a vector bundle $\pi : E \rightarrow M$, such that the space of sections $\Sec(E)$ is a  para-associative $C^\infty(M)$-algebra. Let $E$ and $E'$ be para-associative algebroids over the same base manifold $M$. A \emph{para-associative algebroid homomorphism} is a vector bundle morphism (over the identity) $\phi : E \rightarrow E'$ that the induced map of sections $\phi^\#: \Sec(E) \rightarrow \Sec(E')$ is a para-associative algebra homomorphism. 
\end{definition}
\begin{remark}
We could refer to para-associative algebroids as \emph{weak para-associative algebra bundles} following the nomenclature of weak associative algebra bundles (see for example \cite{Canlubo:2019}). However, for brevity and our perspective of horizontal categorification, we prefer algebroids.
\end{remark}
To spell some of this out, we have a para-associative product that is $C^\infty(M)$-linear on the sections of a vector bundle, i.e., 
\begin{equation}
\big[ [u_1,u_2,u_3] , u_4,u_5 \big] = \big [ u_1,[u_4,u_3,u_2],u_5\big] =  \big[ u_1,u_2,[u_3 , u_4,u_5] \big]\, ,
\end{equation}
and
\begin{equation}\label{eqn:ModStru}
 [fu, v, w] =[u, fv, w] = [u, v, fw] = f[u, v, w]\,,
 \end{equation}
for all $u_i, u, v,w \in \Sec(E)$ and $f\in C^{\infty}(M)$.  Note that the zero section is an absorbing element, i.e., the zero of the algebra.\par 
On every fibre of a para-associative algebroid there is the structure of a para-associative algebra which is constructed as follows: Let $m \in M$ and let $\alpha, \beta$ and $\gamma \in E_m$ be vectors in the fibre over $m$. We can choose three sections $u,v$ and $w\in \Sec(E)$ such that $u(m)= \alpha$, $v(m) = \beta$ and $w(m)= \gamma$. We then define 
\begin{equation}\label{eqn:FibreAlg}
[\alpha, \beta, \gamma] := [u,v,w](m)\,.
\end{equation}
We must establish that this definition is independent of the choice of $u,v$ and $w \in \Sec(E)$. Recall that two sections have the same germ at a point if they agree on a neighbourhood of that point. Let $\bar{w}\in \Sec(E)$ be another section such that $\bar{w}(m) = \gamma$. Then we know that there is an open neighbourhood $U \subseteq M$ of $m\in M$, such that $w|_U = \bar{w}|_U$. In other words, the two sections locally look the same. Thus, to establish the independence of the ternary product on a fibre of the choice of sections, we need to establish that the product is local. The locality is almost immediate from \eqref{eqn:ModStru}.
\begin{lemma}
Let $(E, [-,-,-])$ be a para-associative algebroid and let $w \in \Sec(E)$ be a section that vanishes on a neighbourhood $U\subseteq M$ of a point $m \in M$. Then
$$[u,v,w](m)=0, \qquad [w,u,v](m) =0, \qquad [v,w,u](m)=0,$$
for all $u,v \in \Sec(E)$.
\end{lemma}
\begin{proof}
We will prove the statement $[u,v,w](m)=0$ under the conditions of the lemma as the other statements follow from the same argument. \par 
Given the data of the lemma, we can pick a bump function $\varphi$ such that $\mathrm{supp}(\varphi) \subset U$ and $\varphi(m)=1$. Then $\varphi w$ vanishes identically as the section is zero on $U$ and $\varphi =0$ outside of $U$.  Then
$$[u,v, \varphi w] =0 = \varphi [u,v,w]\,.$$
Thus, evaluating at $m$ we obtain $[u,v,w](m)=0$.
\end{proof}
As a direct consequence, we have the following.
\begin{proposition}\label{prop:FibWellDef}
Let $(E, [-,-,-])$ be a para-associative algebroid. The para-associative algebra on any fibre $E_m$ ($m\in M$) defined as
$$ [\alpha, \beta, \gamma] := [u,v,w](m)\,,$$
is independent of the choice of sections $u,v$ and $w\in \Sec(E)$. 
\end{proposition}
From the definition of a para-associative algebroid, there is no reason why the para-associative algebras on each fibre should be isomorphic. 
Let us take a trivialising open $U \subset M$, i.e., $\pi^{-1}(U) \rightarrow U \times \mathit{V}$,  then  any section of para-associative algebroid can be locally written using a (local) basis $\sigma_\alpha$,
$$u|_U = u^\alpha \sigma_\alpha\,,$$
where $u^\alpha \in C^\infty(U)$. The local property of the ternary product and the $C^\infty(M)$-algebra property allows us to write
\begin{equation}
[u,v,w]|_U = u^\alpha v^\beta w^\gamma \,[\sigma_\alpha, \sigma_\beta, \sigma_\gamma] = u^\alpha v^\beta w^\gamma \,C^\lambda_{\alpha \beta \gamma}\sigma_\lambda\,.
\end{equation} 
We refer to $C^\lambda_{\alpha \beta \gamma}$ as the \emph{structure functions} of the para-associative algebroid. A quick calculation shows that the para-associative property leads to the following.
\begin{proposition}
Let $(E, [-,-,-])$ be a para-associative algebroid. The structure functions satisfy
$$C^\eta_{\alpha \beta \gamma} C^\lambda_{\eta \delta \epsilon} = C^\eta_{\delta  \gamma \beta} C^\lambda_{\alpha \eta \epsilon} = C^\eta_{\gamma \delta \epsilon} C^\lambda_{\alpha \beta \eta}\,.$$
\end{proposition}
\begin{example}
Any vector bundle can be considered a para-associative algebroid by defining the ternary product to be a zero map.
\end{example}
\begin{example}
Any real para-associative algebra can be considered as a para-associative algebroid over a point.
\end{example}
\begin{example}\label{exa:ParaAssAlgR2}
Consider the trivial vector bundle $E = M \times \R^2$. A global basis of sections is provided by any basis of $\R^2$. Thus, picking the standard basis $\{e_1,e_2\}$ we can generalise Example \ref{exm:ParaAssR2}  and provide $M \times \R^2$ with the structure of a para-associative algebroid. 
\end{example}
\begin{example}\label{exa:RiemMan1}
Let $(M,g)$ be a (pseudo-)Riemannian manifold. Then $\sT M$ is a para-associative algebroid with the ternary product being 
$$[X,Y,Z] := g(X,Y)Z\,,$$
for all $X,Y$ and $Z\in \Vect(M)$.  An isometry between two (pseudo-)Riemannian manifolds $(M,g)$ and $(N,h)$  is a diffeomorphism $\phi : M \rightarrow N$, such that $g(X,Y) = h(\phi_* X,\phi_* Y)$. Then a quick calculation shows that an isometry is also a para-associative algebroid homomorphism.
\end{example}
\begin{remark}
Example \ref{exa:RiemMan1} does not rely on the metric being non-degenerate. Thus, we can still construct a para-associative algebroid on the tangent bundle of a manifold with a `degenerate metric'.  For example, the tangent bundle of a Carrollian spacetime is a para-associative algebroid (see \cite{Duval:2014}).
\end{remark}
\begin{example}
Let $(M,g)$ be a (pseudo-)Riemannian manifold. Then $\sT^* M$ is a para-associative algebroid with the ternary product being 
$$[\omega,\eta,\sigma] := g^{-1}(\omega, \eta)\sigma\,,$$
for all $\omega, \eta$ and $\sigma \in \Omega^1(M)$. This ternary product could also have been defined using Example \ref{exa:RiemMan1} and the musical isomorphisms. 
\end{example}
\begin{example}\label{exm:AssAlgJordan}
Consider an associative algebroid, i.e., a vector bundle $\pi : E \rightarrow M$ whose sections form an associative $C^\infty(M)$-algebra. We will denote the binary product by  `dot'.   If the algebra satisfies the `Jordan-like' property $u\cdot v\cdot w = w\cdot v\cdot u$, for all $u,v,w \in \Sec(E)$, then 
$$[u,v,w] :=  u\cdot v\cdot w\,,$$
provides the structure of a commutative para-associative algebroid. For example, commutative associative algebroids are also para-associative algebroids with respect to the triple product. 
\end{example}
\begin{example}
If $E$ is a para-associative algebroid, then $E^{\mathrm{op}}$, defined as the same vector bundle but now with $[u,v,w]^{\textrm{op}} := [w,v,u]$, is also a para-associative algebroid. As a specific example, given the para-associative algebroid defined over a (pseudo-)Riemannian manifold, see Example \ref{exa:RiemMan1}, the opposite structure is
$$[X,Y,Z]^{\textrm{op}} := [Z,Y,X] = g(Z,Y)X\,.$$
\end{example}
\begin{coexample}
Let $(M, \omega)$ be an almost symplectic manifold. We can define a ternary product on $\Vect(M)$,
$$[X,Y,Z] := \omega(X,Y)Z\,,$$
for all $X,Y$ and $Z \in \Vect(M)$. However, the antisymmetry property $\omega(X,Y) = - \omega(Y,X)$, spoils the para-associative property. Thus, we do not obtain a para-associative algebroid structure on the tangent bundle of an almost symplectic manifold.
\end{coexample}
\begin{small}
\noindent \textbf{Aside:}
The definition of a para-associative algebroid (Definition \ref{def:ParaAlgBroid}), directly generalises to a para-associative superalgebroid. That is, a vector bundle $\pi : E \rightarrow M$ in the category of supermanifolds (see for example \cite{Carmeli:2011}) that comes with a para-associative $C^\infty(M)$-superalgbera on its sections.  The minor modifications are some signs:
\begin{equation*}
\big[ [u_1,u_2,u_3] , u_4,u_5 \big] = (-1)^{\widetilde{u_2}(\widetilde{u_3} +\widetilde{u_4})+ \widetilde{u_3} \,\widetilde{u_4}} \,\big [ u_1,[u_4,u_3,u_2],u_5\big] =  \big[ u_1,u_2,[u_3 , u_4,u_5] \big]\, ,
\end{equation*}
and
\begin{equation*}
 [fu, v, w] =(-1)^{\widetilde{f}\,\widetilde{u} }\, [u, fv, w] = (-1)^{\widetilde{f}(\widetilde{u}+ \widetilde{v})}\,[u, v, fw] = f[u, v, w]\,,
 \end{equation*}
for all homogeneous $u_i, u, v,w \in \Sec(E)$ and $f\in C^{\infty}(M)$ Here `tilde' represents the $\Z_2$-degree of the objects. Extension to non-homogeneous elements is via linearity. 
\begin{example}
Let $(M,g)$ be an even Riemannian supermanifold. Then $\sT M$ is a para-associative superalgebroid with the ternary product being 
$$[X,Y,Z] = g(X,Y)Z\,,$$
for all $X,Y$ and $Z\in \Vect(M)$.
\end{example}
\end{small}
\subsection{From Para-associative Algebroids to Associative Algebroids}
Given a semiheap, we can construct a semigroup by selecting an element of the semiheap. However, we do not have an equivalence of categories: different choices of an element can lead to non-isomorphic semigroups, and the reverse process of constructing a semiheap from an arbitrary semigroup may not exist. In the current setting, we have the following result.
\begin{proposition}\label{prop:ParaAssToAssAlg}
Let $(E, [-,-,-])$ be a para-associative algebroid. Then
$(E,u\star_e v)$ is an associative algebroid where
  $u\star_e v := [u,e,v]$, and $e \in \Sec(E)$ is a selected section. If the ternary product is commutative, then the binary product is commutative. Furthermore, if the chosen element $e$ is biunital, then the associative $C^\infty(M)$-algebra on $\Sec(E)$ is unital and the unit is $e$.
\end{proposition}
\begin{proof} It is clear that the binary operation defined is $C^\infty(M)$-linear as the ternary product is. We need to establish the other properties. 
\begin{description}
\setlength\itemsep{1em}
\item[Associativity] Observe that  $(u\star_e v)\star_e w = [[u,e,v], e, w]$ and  $u\star_e( v\star_e w) = [u,e,[v, e, w]]$. Thus para-associativity implies associativity.
\item[Commutativity] Assuming that the ternary product is commutative then $u\star_e v = [u,e,v] = [v,e,u] = v\star_e u$.
\item[Unital] Assuming $e$ is biunital, we observe that $u\star_e e = [u,e,e]=u$ and $e\star_e u = [e,e,u]=u$.
\end{description}
\end{proof}
\begin{remark}
Proposition \ref{prop:ParaAssToAssAlg} does not imply an equivalence of categories of para-associative algebroids and associative algebroids. Although ternary and binary structures are related, they are not equivalent. 
\end{remark}
\begin{example}
Let $(M,g, T)$ be a spacetime, i.e., a Lorentzian manifold $(M,g)$, equipped with a nowhere vanishing timelike vector field $T$.  Then there is an associative $C^\infty(M)$-algebra on $\Vect(M)$ (see Example \ref{exa:RiemMan1}) where the binary product is given by
$$X\star_T Y := g(X,T)Y\,,$$
for all $X,Y \in \Vect(M)$.
\end{example}
\begin{proposition}
Let $(E,[-,-,-])$ be a para-associative algebroid. Furthermore, let $e,e'\in \Sec(E)$ be biunital sections. Then the associative $C^\infty(M)$-algebras $(\Sec(E), \star_e)$ and  $(\Sec(E), \star_{e'})$ are canonically isomorphic.
\end{proposition}
\begin{proof}
We claim that the isomorphism is given by
\begin{align*}
\psi : \,& (\Sec(E), \star_e) \rightarrow (\Sec(E), \star_{e'})\\
& u \longmapsto [u, e , e']\,.
\end{align*}
This map is clearly $C^\infty(M)$-linear. A quick calculation
$$\psi(u \star_e v) = [[u,e,e'],[e,e,e'], [v, e,e']] = [[u,e,e'],e', [v, e,e']]= \psi(u) \star_{e'}\psi(v) \,, $$
shows that we have a homomorphism of associative $C^\infty(M)$-algebras. The inverse map is given by
\begin{align*}
\psi^{-1} : \,& (\Sec(E), \star_{e'}) \rightarrow (\Sec(E), \star_e)\\
& u \longmapsto [u, 'e , e']\,.
\end{align*}
The reader can check that this is the correct inverse and moreover, the condition that both $e$ and $e'$ be biunits is essential.
\end{proof}
\begin{example}
Consider $E = M \times \R^2$ with the para-associative algebroid structure given by Example \ref{exa:ParaAssAlgR2} (following Example \ref{exm:ParaAssR2}).  As both $e_1$ and $e_2$ are biunital elements (thought of as constant sections), setting $e = e_1$ and $e =e_2$ produces isomorphic associative $C^\infty(M)$-algebras. Explicitly,
$$u \star_{e_1} v = (u^1 v^1 +u^2 v^2)e_1 + (u^1 v^2 + u^2 v^1)e_2\,, \qquad u \star_{e_2} v =  (u^1 v^2 + u^2 v^1)e_1 + (u^1 v^1 +u^2 v^2)e_2 \,. $$ 
The binary products here are defined by the two representations of the cyclic group $C_2$ by the canonical basis of $\R^2$.  
\end{example}
\begin{coexample}
Let $(M, g)$ be a Riemannian manifold with $M$ being closed, connected and of Euler characteristic zero. On such manifolds, nowhere vanishing vector field exists and can be globally normalised, $g(X,X) = 1$. This normalisation is not possible for vector fields that have singularities.  Observe that a nowhere-vanishing normalised vector field is a left-biunit, i.e.,
$$[X,X, Y] = g(X,X)Y = Y\,.$$
However, such vector fields are not right-biunits, i.e.,
$$[Y, X,X]= g(Y,X)X \neq Y\,,$$
for all $Y \in \Vect(M)$. Thus, while we can construct associative $C^\infty(M)$-algebra for every nowhere vanishing normalised vector field, they are not all canonically isomorphic.
\end{coexample}
 \begin{small}
\noindent \textbf{Aside:} Given an associative algebroid, we can construct a totally intransitive Lie algebroid via the commutator of sections, i.e.,
$$[u,v] := uv - vu\,.$$
We can similarly start with a para-associative algebroid and consider the ternary commutator
$$\{u,v,w\} := [u,v,w] -[v,u,w] + [w,u,v]- [u,w,v]+ [v,w,u] - [w,v,u]\,.$$
However, the ternary commutator does not, in general, satisfy the Filippov-Jacobi identity. Thus we do not obtain a (non-trivial) totally transitive Filippov algebroid as defined by Grabowski and Marmo \cite{Grabowski:2000}. Examining Example \ref{exa:RiemMan1} and Example \ref{exm:AssAlgJordan}, we observe that the ternary commutator identically vanishes in these example.  
\end{small}
%
\subsection{New Para-associative Algebroids from Old}
\subsubsection*{Direct Sum of Para-associative algebroids}
Given two vector bundles $\pi_1 : E_1 \rightarrow M$ and $\pi_2 : E_2 \rightarrow M$ over the same base, the direct sum is defined as 
$$E_1 \oplus_M E_2 :=  \big \{(p,q)\in E_1 \times E_2 ~ |~ \pi_1(p)= \pi_2(q)  \big \}\,, $$
and the projection $\tau : E_1 \oplus_M E_2 \rightarrow M$ is defined as $\tau(p,q)= \pi_1(p) = \pi_2(q)$. At the level of sections $\Sec(E_1)\oplus_{C^\infty}\Sec(E_2) = \Sec(E_1)\times \Sec(E_2)$, as sets and the module structure is defined component wise
\begin{align*}
& (u_1, u_2)+ (v_1 , v_2) = (u_1 + v_1, u_2 + v_2)\,,\\
& f(u,v) = (fu,fv)\,,
\end{align*}
for all $u_i,u \in \Sec(E_1)$ and $v_i,v \in \Sec(E_2)$ and $f \in C^\infty(M)$. If $E_1$ and $E_2$ are para-associative algebroids, then $E_1 \oplus_M E_2$ is a para-associative algebroid where the ternary product is defined component-wise, i.e.,
$$[(u_1, u_2), (v_1, v_2), (w_1,w_2)] := ([u_1, v_1, w_1], [u_2, v_2, w_2] )\,.$$
\subsubsection*{Tensor Product of Para-associative Algebroids}
Unlike the direct sum construction, we need three vector bundles $\pi_1 : E_1 \rightarrow M$, $\pi_2 : E_2 \rightarrow M$ and $\pi_3 : E_3 \rightarrow M$ over the same base.  The triple tensor product is the vector bundle $\tau : E_1 \otimes E_2 \otimes E_3 \rightarrow M$, where the fibres are the tensor product of the component fibres, i.e., $(E_1 \otimes E_2 \otimes E_3)_m \simeq (E_1)_m \otimes (E_2)_m \otimes (E_3)_m $, for all $m \in M$. The projection is defined as $\tau(p\otimes q \otimes r) = \pi_1(p) = \pi_2(q) = \pi_3(r)$. At the level of sections
$$\Sec(E_1 \otimes E_2 \otimes E_3) \simeq \Sec(E_1)\otimes_{C^\infty(M)} \Sec(E_2)\otimes_{C^\infty(M)} \Sec(E_3)\,.$$
To spell this structure out, 
\begin{align*}
& u\otimes v \otimes (w_1+ w_2) = u \otimes v \otimes w_1 + u \otimes v \otimes w_2\,,\\
& u \otimes (v_1 + v_2) \otimes w = u \otimes v_1 \otimes w + u \otimes v_2 \otimes w\,,\\
& (u_1 +u_2)\otimes v \otimes w =  u_1 \otimes v \otimes w +  u_2 \otimes v \otimes w\,,
\end{align*}
and 
$$fu \otimes v \otimes w = u \otimes fv \otimes w =u \otimes v \otimes fw\,.$$
If $E_1, E_2$ and $E_3$ are para-associative algebroids then the tensor product $E_1 \otimes E_2 \otimes E_3$ is also para-associative algebroid with the ternary product being 
$$[(u_1 \otimes u_2 \otimes u_3),(v_1 \otimes v_2 \otimes v_3) , (w_1 \otimes w_2 \otimes w_3)] = [u_1, v_1, w_1]\otimes [u_2, v_2, w_2] \otimes [u_3, v_3, w_3]\,.$$
The para-associativity of each component ensures that the tensor product ternary product is para-associative.
%
\subsection{Local Trivialisations}
The definition of a para-associative algebroid does not say anything about a local trivialisation, i.e., each fibre para-associative algebra need not be isomorphic to some standard para-associative algebra. 
\begin{example}
Consider the trivial vector bundle  $\pi : (0,1)\times \mathit{V}\rightarrow (0,1)$, where $\mathit{V}$ is a finite dimensional vector space. We equip the sections of this vector bundle with a para-associative product
$$[u,v,w] := t B(u,v)w\,,$$
where $t$ is the coordinate on $(0,1)$ and $B$ is a symmetric bilinear form on $\mathit{V}$. We have a trivial vector bundle, but note that the fibre para-associative algebras are not all isomorphic. Specifically,  every para-associative algebra on $\mathit{V}_t$ for $0 < t< 1$ is isomorphic to the para-associative algebra on $\mathit{V}_1$, explicitly, $u \mapsto t^{-1/3}u$. However, the algebra on $\mathit{V}_0$ is the zero ternary algebra, which is not isomorphic to the para-associative algebra on $\mathit{V}_1$.
\end{example}
In this subsection, we will address the question of local triviality. We first propose the following definition by generalising the notion of an associative algebra bundle (see for example \cite{Canlubo:2019}).
\begin{definition} A \emph{para-associative algebra bundle} is a vector bundle $\pi  : E \rightarrow M$ with the structure of a para-associative algebra on each fibre $E_m$, $m \in M$, such that for each point of $M$ there exists an open neighbourhood $U\subseteq M$ and a para-associative algebra $\mathcal{A}$ and a diffeomorphism $\psi : \pi^{-1}(U) \stackrel{\sim}{\rightarrow} U \times \mathcal{A}$, that preserves the fibres and restricts to a para-associative algebra isomorphism $\psi_n : E_n \stackrel{\sim}{\rightarrow} \mathcal{A}$ in each fibre over  $n \in U$.
\end{definition}
Given a para-associative algebra bundle, we can define a para-associative algebroid in the standard way by defining point-wise operations. In particular, for three sections $u,v,w \in \Sec(E)$ we define the ternary product via
$$[u,v,w](m) :=  [u(m), v(m), w(m)]\in E_m \simeq \mathcal{A}\,,$$
for any $m \in M$. This construction is consistent with passing from the global algebra to the fibre algebras, see Proposition \ref{prop:FibWellDef}. \par 
\begin{lemma}\label{lem:IsoFibBun}
Let $(E, [-,-,-])$ be a para-associative algebroid over a connected base manifold $M$. If all the fibre para-associative algebras are isomorphic, then $\pi :E\rightarrow M$ is a para-associative algebra bundle. 
\end{lemma}
\begin{proof}
Assume all the fibres are isomorphic as para-associative algebras.  Thus, for any $m \in M$, we can choose isomorphisms $\lambda_m: E_m \stackrel{\sim}{\rightarrow} \mathcal{A}$, for some fixed para-associative algebra $\mathcal{A}$. Note that these isomorphisms are linear and so smooth. From the vector bundle structure, there is an open neighbourhood of $m$, $U\subset M$, such that there is a diffeomorphism  $\psi: \pi^{-1}(U)\stackrel{\sim}{\rightarrow} U \times \mathit{V}$, where $\mathit{V}$ is a finite-dimensional vector space. This map restricts to a linear isomorphism $\psi_n : E_n \stackrel{\sim}{\rightarrow}\mathit{V}$, for all points $n \in U$.\par 
We can replace $\mathit{V}$ with $\mathcal{A}$ (which are obviously isomorphic as vector spaces) by defining 
$$\bar{\psi}_n := (\psi_n\circ \lambda_n)^{-1}\circ \psi_n : E_n \stackrel{\sim}{\rightarrow} \mathcal{A}\,,$$
which by construction is an isomorphism between the two para-associative algebras.  The diffeomorphism $\bar{\psi} : \pi^{-1}(U)\stackrel{\sim}{\rightarrow} U \times \mathcal{A}$ is given by point-wise composition with the isomorphism of fibres to $\mathcal{A}$. Specifically, if $\psi_m(p) = (m, \alpha) $, then $\bar{\psi}_m(p) = (m, \lambda_m(\alpha))$.
\end{proof}
In the associative setting, Canlubo \cite[Proposition 3]{Canlubo:2019} showed that the existence of a particular kind of linear connection on an associative algebroid implies local triviality. Here we will generalise this to the para-associative setting.
\begin{definition}\label{def:DiffCon}
Let $(E, [-,-,-])$ be a para-associative algebroid. A linear connection on the vector bundle $\pi:E \rightarrow M$, $\nabla : \Vect(M)  \times \Sec(E)\times \Sec(E) \rightarrow \Sec(E)$, is said to be a \emph{differential connection} if it satisfies
$$\nabla_X [u,v,w] =  [\nabla_X u,v,w] + [u,\nabla_X v,w] + [u,v,\nabla_X w]\,,$$
for all $X \in \Vect(M)$ and $u,v,w \in \Sec(E)$.
\end{definition}
The nomenclature `differential' in Definition \ref{def:DiffCon} is due to the connection satisfying the ternary Leibniz rule. \par 
Using local coordinates $x^a$, and defining as standard $\nabla_a\sigma_\alpha :=  \Gamma_{a \alpha}^\beta \sigma_\beta$, the differential condition is expressed as
$$\partial_a C_{\alpha\beta\gamma}^\lambda + C_{\alpha\beta\gamma}^\epsilon \Gamma_{a \epsilon}^\lambda =  \Gamma_{a \alpha}^\epsilon C_{\epsilon\beta\gamma}^\lambda +  \Gamma_{a \beta}^\epsilon C_{\alpha\epsilon\gamma}^\lambda +  \Gamma_{a \gamma}^\epsilon C_{\alpha\beta\epsilon}^\lambda\,.$$
\begin{remark}
The Riemannian curvature tensor is defined as 
$$R_\nabla(X,Y)(-) := [\nabla_X,\nabla_Y ](-) - \nabla_{[X,Y]}(-)\,.$$
A quick calculation shows that if the connection is a differential connection, then the Riemannian curvature tensor is a derivative, i.e.,
$$R_\nabla(X,Y)[u,v,w] = [R_\nabla(X,Y)u,v,w] + [u,R_\nabla(X,Y)v,w] + [u,v,R_\nabla(X,Y)w]\,.$$
\end{remark}
\begin{example}
By considering a para-associative algebra as a para-associative algebroid over a point, then a differential connection is precisely a ternary differential (see \cite{Carlsson:1976}). 
\end{example}
\begin{example}
Let $(M, g)$ be a (pseudo-)Riemannian manifold and let $\nabla$ be a metric connection, i.e., 
$$\nabla_X\big(g(Y,Z)\big) = g \big(\nabla_X Y, Z \big) + g \big( Y, \nabla_XZ \big)\,.$$
A quick calculation shows that 
$$\nabla_W[X,Y,Z] = [\nabla_W X,Y,Z] + [X,\nabla_W Y,Z]  + [X,Y,\nabla_W Z]\,,$$
for all $W, X,Y,Z \in \Vect(M)$. Thus, any metric connection is a differential connection. Moreover, one can directly observe that if a connection is a differential connection, then it is also metric compatible.  
\end{example}
\begin{lemma}\label{lem:DiffConIso}
Let $E$ be a para-associative algebroid over $M$. If the base manifold $M$ is connected and $E$ admits a differential connection $\nabla$, then the para-associative algebras on each fibre are isomorphic.
\end{lemma}
\begin{proof}
Assume $E$ comes equipped with a ternary connection $\nabla$. Let $m, n \in M$ be two arbitrary but distinct points, and let $c: I  \rightarrow M$ be a smooth path in $M$ such that $c(0)= m$ and $c(1) = n$.  We need connectedness to guarantee that any two points can be connected by a path. Given $\alpha, \beta,\gamma \in E_m$, parallel transport of these vectors is the unique extension to the vectors to parallel sections along $c$. That is we have sections $u,v$ and $w$ that satisfy 
\begin{align*}
& \nabla_{c'}u =0\,,& u(m)= \alpha\,,\\
& \nabla_{c'}v =0\,,& v(m)= \beta\,,\\
& \nabla_{c'}w =0\,,& u(m)= \gamma\,.\\
\end{align*}
By definition/construction $[\alpha, \beta, \gamma] = [u,v,w](m)$.  As we have a ternary connection 
$$\nabla_{c'} [u,v,w] =  [\nabla_{c'} u,v,w] + [u,\nabla_{c'} v,w] + [u,v,\nabla_{c'} w] = 0\,.$$
Thus, $[u,v,w]$ is a unique section along $c$, and it defines the parallel transport of $[\alpha,\beta, \gamma]$. The parallel transport operator $\Phi(c)_m^n : E_m \rightarrow E_n$ is a linear isomorphism. We need to show that it is a para-associative algebra morphism, so all we need to establish is that the ternary products are respected. Explicitly, we have 
\begin{align*}
& \Phi(c)_m^n(\alpha) :=u(n) =: \alpha'\,,\\ 
& \Phi(c)_m^n(\beta) :=v(n) =: \beta'\,,\\
& \Phi(c)_m^n(\gamma) :=w(n) =: \gamma'\,.
\end{align*}
Furthermore,
$$\Phi(c)_m^n([\alpha,\beta,\gamma]) = [u,v,w](m) \stackrel{\mathrm{def}}{=} [\alpha',  \beta', \gamma'] = [\Phi(c)_m^n(\alpha), \Phi(c)_m^n(\beta), \Phi(c)_m^n(\gamma)]\,.
$$
Thus, $\Phi(c)_m^n$ is an isomorphism of para-associative algebras.  Thus, for any and all pairs of distinct points $m$ and $n \in M$, we have an isomorphism of the fibre para-associative algebras. 
\end{proof}
\begin{remark}
If the base manifold $M$ is not connected, then Lemma \ref{lem:DiffConIso} needs slight modification to say that all the fibre para-associative algebras over each connected component are isomorphic. 
\end{remark}
\begin{theorem}\label{thm:DiffConBundle}
Let $E$ be a para-associative algebroid over $M$. If the base manifold $M$ is connected and if $E$ can equipped with a differential connection $\nabla$, then $E$ is also a para-associative algebra bundle.
\end{theorem}
\begin{proof} Assume that $E$ is equipped with a differential connection. Lemma \ref{lem:DiffConIso} tells us that if we have a differential connection on $E$, then the fibre para-associative algebras are all isomorphic. Lemma \ref{lem:IsoFibBun} establishes that if fibre para-associative algebras are isomorphic, then we have a para-associative algebra bundle.
\end{proof}
The existence of a differential connection is a sufficient condition for a para-associative algebroid to be locally trivial, i.e., a para-associative algebra bundle. If we can construct a differential connection on an arbitrary para-associative algebra bundle, then we have a necessary condition. Of course, any vector bundle can be equipped with a linear connection, the question is if we can construct a differential connection.\par 
Consider the trivial case $E = M \times \mathcal{A}$. Sections are of the form $u = u^\alpha \sigma_\alpha$, where $u^\alpha \in C^\infty(M)$ and $\sigma_\alpha$ is a basis of $\mathcal{A}$. The ternary product is given by
$$[u,v,w] =  u^\alpha v^\beta w^\gamma \, C^\lambda_{\alpha \beta \gamma}\sigma_\lambda\,,$$
where $C^\lambda_{\alpha \beta \gamma}$ is constant (we have the same para-associative algebra for every point). We define the trivial connection via $\nabla_X \sigma_\alpha =0$, for all $X \in \Vect(M)$. A key observation is that this trivial connection is a differential connection, i.e., 
$$\nabla_X [u,v,w] =  [\nabla_X u,v,w] + [u,\nabla_X v,w] + [u,v,\nabla_X w]\,.$$
\begin{proposition}\label{prop:ParaBunConAlways}
Let $\pi : E\rightarrow M$ be a para-associative algebra bundle over a connected base $M$. Then $E$, considered as a para-associative algebroid, always admits a differential connection. 
\end{proposition} 
\begin{proof} The proof follows the standard arguments of constructing a connection. Briefly, given a trivialising cover $\{U_i\}_{i \in \mathcal{I}}$, on each $\pi^{-1}(U_i)\stackrel{\sim}{\rightarrow}  U_i  \times \mathcal{A}$ we take the local connection to be the trivial connection, $\nabla^i$.  Then the connection is defined as $\nabla :=  \sum \rho_i \nabla^i$, where $\sum  \rho_i$ is an appropriately chosen locally finite partition of unity. The differential property follows as each local connection is a differential connection.
\end{proof}
Theorem \ref{thm:DiffConBundle} together with Proposition \ref{prop:ParaBunConAlways}, led to the following.
\begin{observation}\label{obs:SuffNessCon}
The existence of a differential connection on a para-associative algebroid is a necessary and sufficient condition for local triviality. 
\end{observation}
\begin{example}
Let $(M,g)$ be a semi-pseudo-Riemannian manifold, i.e., we allow the metric $g$ to be of non-constant signature. Note that if the signature changes, the metric will be degenerate over some region(s) of $M$, and so cannot have constant rank. None-the-less, we can construct a para-associative algebroid structure on $\sT M$ following the pseudo-Riemannian case, i.e.,
$$[X,Y,Z] := g(X,Y)Z\,,$$
for all $X,Y,Z\in \Vect(M)$. The question of local triviality can be addressed. Assuming that $g$ does indeed change signature across $M$, then \cite[Theorem 3.1]{Iliev:2002} tells us that there are no (globally defined) affine connections that are metric compatible. Thus, we conclude that we do not have a locally trivial structure, i.e., we do not have a para-associative algebra bundle, as we cannot construct a differential connection. 
\end{example}
\begin{small}
\noindent \textbf{Aside:} semi-pseudo-Riemannian manifolds are motivated by quantum gravity and cosmology.  The use of metrics without a fixed signature can be traced back to atleast Hartle \& Hawking \cite{Hartle:1983}, and more explicitly Sakharov \cite{Sakharov:1984}. While there is a lot to say here, and the literature has significantly grown since the 1980s, details lie outside the scope of this note. 
\end{small}
%
%
\section*{Acknowledgements}  
The author thanks Steven Duplij and Janusz Grabowski  for their help and encouragement.

%

\end{document}